\pdfoutput=1
\documentclass[11pt]{article}

\usepackage[top=3.2cm,bottom=2.5cm]{geometry}
\usepackage{graphicx}
\usepackage{mathptmx}      
\usepackage{amsmath,amssymb,amsfonts,amsthm}
\usepackage{stmaryrd}
\usepackage{eucal}
\usepackage{xcolor}
\usepackage{mathtools}
\usepackage{microtype}
\usepackage{enumerate}
\usepackage{url}
\newcommand{\change}[1]{#1}

\newtheorem{theorem}{\bf Theorem}

\newtheorem{example}{\bf Example}

\newcommand{\R}{{\mathbb{R}}}

\newcommand{\bu}{{\mathbf{u}}}
\newcommand{\bv}{{\mathbf{v}}}
\newcommand{\hbu}{{\hat{\bu}}}

\newcommand{\bA}{{\mathbf{A}}}
\newcommand{\bI}{{\mathbf{I}}}
\newcommand{\bb}{{\mathbf{b}}}
\newcommand{\bH}{{\mathbf{H}}}
\newcommand{\bL}{{\mathbf{L}}}

\newcommand{\abs}[1]{{\left\lvert #1 \right\rvert}}
\newcommand{\mult}{{R}}
\newcommand{\domain}{{\mathcal{D}}}
\DeclareMathOperator{\rank}{rank}

\begin{document}

\title{On low-rank approximability of solutions to high-dimensional operator equations and eigenvalue problems}

\author{ Daniel Kressner\thanks{ANCHP, MATHICSE, EPF Lausanne, Switzerland. Email: {\tt daniel.kressner@epfl.ch}} \and 
    Andr\'e Uschmajew\thanks{Hausdorff Center for Mathematics \& Institute for Numerical Simulation,
University of Bonn, 53115 Bonn, Germany.
E-mail: {\tt uschmajew@ins.uni-bonn.de}}
}

\date{}

\maketitle

\begin{abstract}
Low-rank tensor approximation techniques attempt to mitigate the overwhelming
complexity of linear algebra tasks arising from high-dimensional applications.
In this work, we study the low-rank approximability of solutions to linear systems
and eigenvalue problems on Hilbert spaces. Although this question is central to the success of all existing
solvers based on low-rank tensor techniques, very few of the results available so far allow to draw meaningful conclusions for higher dimensions.
In this work, we develop a constructive framework to study low-rank approximability.
One major assumption is that the involved linear operator admits a low-rank representation with respect to the chosen
tensor format, a property that is known to hold in a number of applications. 
Simple conditions, which are shown to hold for a fairly general problem class, guarantee that our  derived low-rank truncation error estimates do not deteriorate as the dimensionality increases.


\end{abstract}

\smallskip

\section{Introduction}

The past few years have seen a growing activity in applying low-rank tensor techniques to the
approximate solution of high-dimensional problems, see, e.g.,~\cite{Grasedyck2013a,Hackbusch2012} for survey.
The success of these techniques crucially depends on the ability to approximate the object of interest
by a tensor of low rank with respect to the chosen tensor format. Although this property has been frequently
confirmed in practice, there is little theoretical insight into this matter so far. 

An important special case of the problems considered in this work are matrix equations
of the form $\bA(U) = B$ for a linear operator $\bA: \R^{M \times N} \to \R^{M \times N}$.
Clearly, any such operator can be written in the form 
\[
 \bA(U) = A^{(1)}_1 U A^{(2)}_1 + A^{(1)}_2 U A^{(2)}_2 + \cdots + A^{(1)}_{r_\bA} U A^{(2)}_{r_\bA}, \qquad A^{(1)}_i \in \R^{M\times M}, \quad 
 A^{(2)}_i \in \R^{N\times N}
\]
for some $r_\bA \le MN$. For $r_\bA = 1$ and invertible 
matrices $A^{(1)}_1\!,\ A^{(2)}_1$ the rank of the solution $U$ equals the rank of $B$.
This property does not hold for $r_\bA \ge 2$ and one then considers the question of low-rank approximability of $U$, that is, the decay of its singular values. Particular attention has been paid to the 
case of a Lyapunov matrix equation
\[
 A U + U A^T = B
\]
for a matrix $B$ of low rank, which plays an important role in control
and model reduction, see, e.g.,~\cite{Benner2013}.
A number of works~\cite{Antoulas2002,Baker2015,Grasedyck2004,Grasedyck2003a,Grubisic2014,Penzl2000,Sabino2006} have been devoted to studying 
low-rank approximability for this problem.
In particular, it has be shown that the singular values of $U$ decay exponentially when
$A$ is symmetric positive definite. All existing proof techniques implicitly
rely on the fact that the two operators $U \mapsto AU$ and $U\mapsto U A^T$ commute. In particular,
this allows for the simultaneous diagonalization of both operators, which greatly simplifies the approximation problem. When this commutativity property is lost, these techniques fail. For example,
only partial results~\cite{Benner2013a,Merz2012} are available so far for the innocently looking modification
\[
 A U + U A^T + CU C^T= B
\]
for general matrix $C$, which plays a role in bilinear and stochastic control.
This indicates that we cannot expect to obtain exponential singular value decay for such generalizations. 

In general, we consider linear systems and eigenvalue problems of the form
\begin{equation} \label{eq:highdimblabla}
  \bA \bu  = \bb, \qquad \bA \bu = \lambda \bu,
\end{equation}
where $\bA$ is a self-adjoint positive definite and bounded linear operator on a tensor product $H_1 \otimes \cdots \otimes H_d$
of Hilbert spaces $H_\mu$, $\mu = 1,\ldots,d$.
We will study the low-rank approximability of the solution $\bu \in H_1 \otimes \cdots \otimes H_d$
in certain tensor network formats, such as the 
tensor train format~\cite{Oseledets2011} (matrix product states~\cite{OestlundRommer1995}) and
the hierarchical Tucker format~\cite{HackbuschKuehn2009}  (tensor tree networks~\cite{Shi2006}).
For these formats, the low-rank approximability is closely tied to the singular value decays of certain bilinear unfoldings associated with the tensor~\cite{Hackbusch2012}. 
This plays an important role in the study of quantum many-body systems~\cite{Schollwock2011}, 
where these decays are reflected in bounds on the entanglement entropy~\cite{EisertCramerPlenio2010}.
For linear lattice models, rigorous bounds by Hastings~\cite{Hastings2007}
imply a low-rank approximability that does \emph{not} deteriorate as the order $d$ increases.
In the special case of frustration-free systems, similar results~\cite{Arad2012}
can be derived via a simplified construction that only takes the algebraic properties of the involved operators
into account.

The purpose of this work is to propose a general framework for obtaining singular value decay estimates for the solutions of~\eqref{eq:highdimblabla}.
Following the basic idea of~\cite{Arad2012}, our results are based on controlling the rank
growth of a fixed-point iteration.
This approach is constructive and only exploits the tensor product structure of the involved operators. The assumed structure features quite frequently in applications, for example in
Schr\"odinger type eigenvalue problems~\cite{Khoromskij2010d,Kressner2011a},
quantum many-body systems with local interactions~\cite{Schollwock2011},
the chemical master equation for simulating biochemical reaction networks~\cite{Kazeev2013},
and Markov models for queuing networks~\cite{Kressner2014a}.
Under certain conditions, the derived estimates do not deteriorate with increasing $d$.
Our construction shares similarities with recent results by Bachmayr and Dahmen~\cite{BachmayrDahmen2015},
who use the method of steepest descent to design a nearly optimal solver for linear systems.
In contrast to our work, these results
\emph{assume} the low-rank approximability of the solution a priori.

Our results state algebraic approximation rates with respect to increasing ranks.
An exponential approximation rate can only be obtained under certain
commutativity assumptions, similar to the Lyapunov equation discussed above.
One of the very few results in this direction is the approximation
of the solution to the $d$-dimensional Poisson equation by means of exponential sums~\cite{Grasedyck2004,Hackbusch2012}.


The rest of this paper is organized as follows. In Section~\ref{sec: abstract results},
we provide a general framework for 
assessing the interplay between rank growth and convergence rate of fixed point iterations
on tensor products of Hilbert spaces.
Section~\ref{sec:linearequations} specializes this framework to the method of steepest descent
applied to linear systems with tensor product structure, resulting in singular value decay estimates
for the solution. In a similar manner, Section~\ref{sec: eigenvectors} covers symmetric eigenvalue problems.

\section{Approximation by fixed-point iterations with finite rank growth}\label{sec: abstract results}

In this section, we develop our general framework for low-rank tensor approximation
by first considering the case $d = 2$ and then extending these results to tensors of arbitrary
order $d$.

\subsection{Bilinear approximation}

Let $H_1, H_2$ be two Hilbert spaces (either both real or both complex), and
consider the tensor product  $\bH = H_1 \otimes H_2$ with the induced inner product
$\langle u_1 \otimes v_1, u_2 \otimes v_2 \rangle_{\bH} = \langle u_1 , v_1 \rangle_{H_1} \cdot \langle u_2 , v_2 \rangle_{H_2}$. Note that $\bH$ is isomorphic to $HS(H_1,H_2)$,
the space of Hilbert-Schmidt operators from $H_2$ to $H_1$.
Every tensor $\bu \in \bH$ admits a \emph{singular value decomposition} (SVD)
\begin{equation}\label{eq:SVD}
\bu = \sum_{k=1}^\infty \sigma_k u_k \otimes v_k,
\end{equation}
with $u_1,u_2,\dots$ and $v_1,v_2,\dots$ forming complete orthonormal systems in $H_1$ and $H_2$, respectively, and \emph{singular values} $\sigma_1 \ge \sigma_2 \ge \dots \ge 0$. The smallest $r$ for which $\sigma_{r+1}=0$ is called the rank of $\bu$. If there is no such $r$, the rank of $\bu$ is $\infty$.

We denote by
\[
\tau_r(\bu) = \inf_{\substack{\tilde u_1,\dots,\tilde u_r \in H_1 \\ \tilde v_1, \dots, \tilde v_r \in H_2}} \bigg\| \bu - \sum_{k=1}^r \tilde u_k \otimes \tilde v_k \bigg\|_{\bH}
\]
the error for the best bilinear approximation of rank at most $r$. It is well known that the infimum is achieved by the sum of the first $r$ terms in the singular value decomposition, and
\[
\tau_r(\bu) = \min_{\rank (\bv) \le r} \| \bu - \bv \|_{\bH} = \bigg( \sum_{k=r+1}^\infty \sigma_k^2 \bigg)^{1/2}.
\]

In the sequel we will be concerned with the case that $\bu$ is implicitly given, e.g., as the solution of an optimization problem that represents a linear operator equation or eigenvalue problem.
The basis of our framework is to approach $\bu$ by a fixed-point iteration
\begin{equation}\label{eq:fixed-point iteration}
\bu_{n+1} = \Phi(\bu_n)
\end{equation}
which has a guaranteed convergence rate, but increases the ranks of the iterates at most by a constant factor in every step. Examples for~\eqref{eq:fixed-point iteration} relevant for linear systems are gradient descent methods, like the Richardson iteration that will be used later on. However, other fixed-point iterations are imaginable wherefore we first keep the setting general. We need the following properties.

\medskip

\begin{enumerate}[(i)]
\item
\emph{Contraction:} 
There exists $0<q<1$ and $c>0$ such that
\begin{equation}\label{A1}
 \| \bu_{n+1} - \bu \|_\bH \le c q^{n+1} \| \bu_0 - \bu \|_\bH \quad \text{for all $n$.} \tag{A1}
\end{equation}
\item
\emph{Finite rank growth:} There exists $\mult > 1$ such that
\begin{equation}\label{A2}
\rank(\bu_{n+1}) \le \mult\cdot \rank(\bu_n)  \quad \text{for all $n$.} \tag{A2}
\end{equation}
\end{enumerate}

\medskip

The missing ingredient is that the starting point $\bu_0$ should have known finite rank. In fact, we will assume that $\rank(\bu_0) \le 1$. The limit point (as well as the other properties) of the iteration may depend on the choice of $\bu_0$ (this will become particularly visible  for the case of eigenvalue problems in Section~\ref{sec: eigenvectors}). We therefore consider a set
\begin{equation} \label{eq:domain}
 \domain \subseteq \{ \bu_0 \vcentcolon \text{the sequence $(\bu_n)$ generated from $\bu_0$ by~\eqref{eq:fixed-point iteration} satisfies~\eqref{A1} and~\eqref{A2}}\},
\end{equation}
and assume
\begin{enumerate}[(iii)]
\item
\emph{Rank-one starting point:} Properties~\eqref{A1} and~\eqref{A2} can be satisfied using a starting point in $\domain$ with rank at most one, that is,
\begin{equation}\label{A4}
 \domain \cap \{ \bu_0 \in \bH \vcentcolon \rank(\bu_0) \le 1 \} \neq \emptyset. \tag{A0}
\end{equation}
\end{enumerate}
Given~\eqref{A4}, one can define the quantity
\[
\pi_1(\bu) = \inf_{\substack{\bv \in \domain \\ \rank(\bv) \le 1}} \|\bv - \bu\|_\bH,
\]
and derive the main result of this section.


\begin{theorem}\label{th: abstract result d2}
The existence of a map $\Phi$ on $\bH$ satisfying~\eqref{A4} implies
\begin{equation}\label{eq: first estimate}
\tau_r(\bu) \le c \pi_1(\bu) \sqrt{\left( 1 - \frac{(1-q^2) (r - \mult^{\lfloor \log_\mult r \rfloor})}{(\mult-1) \mult^{\lfloor \log_\mult r \rfloor} } \right)} q^{\lfloor \log_\mult r \rfloor}.
\end{equation}
Simplified bounds are given by
\begin{equation}\label{eq: cleaner bilinear estimate}
\begin{aligned}
\tau_r(\bu) \le c \pi_1(\bu)  q^{\lfloor \log_\mult r \rfloor}  \le c \pi_1(\bu)  q^{(\log_\mult r) - 1} = c \pi_1(\bu) q^{-1} \left( \frac 1 r \right)^{\abs{\frac{\ln q}{\ln \mult}}}.
\end{aligned}
\end{equation}
\end{theorem}


\begin{proof}
For brevity, we write $\tau_r$ instead of $\tau_r(\bu)$. By~\eqref{A4}, there is a starting point $\bu_0 \in \domain$ of rank at most one such that the sequence $(\bu_n)$ formed by~\eqref{eq:fixed-point iteration} satisfies~\eqref{A1} and~\eqref{A2}. Consequently,
$\rank( \bu_n ) \le \mult^n$ and 
\[
\tau_{\mult^n} \le \| \bu_n - \bu\|_{\bH} \le c q^{n} \| \bu_n - \bu_0 \|_\bH.
\]
As this holds for all admissible $\bu_0$, we may pass to the infimum:
\begin{equation}\label{eq: estimate for powers of mult}
\tau_{\mult^n} \le c \pi_1(\bu) q^{n}.
\end{equation}
Since the sequence $(\sigma_k)$ is decreasing, we have
for every $0 \le s < \mult^{n+1} - \mult^n$ that
\[
\sum_{k = \mult^n + 1}^{\mult^n + s} \sigma_k^2 \ge \frac{s}{\mult^{n + 1} - \mult^n} \sum_{k = \mult^n + 1}^{\mult^{n+1}} \sigma_k^2 = \frac{s}{(\mult-1)\mult^n} ( \tau_{\mult^n}^2 - \tau_{\mult^{n+1}}^2 ).
\]
Hence, using~\eqref{eq: estimate for powers of mult},
we obtain for $r = \mult^n + s$ the estimate
\begin{align*}
\tau^2_r = \tau_{\mult^n}^2 -  \sum_{k = \mult^n + 1}^{\mult^n + s} \sigma_k^2 &\le \tau_{\mult^n}^2 - \frac{s}{(\mult-1)\mult^n} ( \tau_{\mult^n}^2 - \tau_{\mult^{n+1}}^2 ) \\
&\le c^2 \pi_1(\bu)^2 \left( 1 - \frac{(1-q^2)s}{(\mult-1)\mult^n} \right) q^{2n},
\end{align*}
as asserted by~\eqref{eq: first estimate}. The simplified bound~\eqref{eq: cleaner bilinear estimate} follows from the observation that the term under the square root in~\eqref{eq: first estimate} is bounded by one.
\end{proof}

By general results for ordered sequences~\cite{DeVore1998}, a decay rate for the tail $\tau_r(\bu)$ yields a decay rate for the singular values themselves. For instance, using~\eqref{eq: cleaner bilinear estimate}, we obtain
\begin{equation}\label{eq: singular value rate}
\sigma_r^2 \le
\frac{\sum_{k = \lfloor r/2 \rfloor + 1}^r \sigma_k^2}{r - \lfloor r/2 \rfloor}
\le \frac{\tau_{\lfloor r/2 \rfloor}^2(\bu)}{\lfloor r/2 \rfloor}
\le c \pi_1(\bu) q^{-2} \left( \frac{1}{\lfloor r/2 \rfloor} \right)^{2\abs{\frac{\ln q}{\ln \mult}}}
\le c \pi_1(\bu) q^{-2} \left( \frac{2}{r-1} \right)^{2\abs{\frac{\ln q}{\ln \mult}}}.
\end{equation}
One consequence of~\eqref{eq: singular value rate} is that the von Neumann entropy of the squared singular values,
\[
S(\bu) = \sum_{k=1}^\infty \sigma_k^2 \log (\sigma_k^2),
\]
remains finite, \change{provided that $q^2 \mult < 1$.  This is a non-trivial result since $\bu \in H_1\otimes H_2$ only implies the convergence of $\sum_{k=1}^\infty \sigma_k^2$. Explicit bounds on the von Neumann entropy $S(\bu)$ are of interest in many applications, for instance in quantum particle models where it represents the \emph{entanglement entropy} of ground states~\cite{Arad2012,EisertCramerPlenio2010,Hastings2007}. The quite strong condition~$q^2 \mult < 1$ on the fixed point iteration will reappear in Theorem~\ref{th: estimate the overlap} to deduce~\eqref{A4} from~\eqref{A1} and~\eqref{A2} in the case that $\domain$ is the affine plane orthogonal to $\bu$~\cite{Arad2012}.}



\subsection{Multilinear approximation} We now consider $d \ge 2$ Hilbert spaces $H_1,H_2,\dots,H_d$ (either all real or all complex). For each subset $t \subseteq \{1,2,\dots,d\}$ of indices with $0<|t|<d$, we have the following isomorphism between the tensor product Hilbert space
\[
\bH = H_1 \otimes H_2 \otimes \dots \otimes H_d
\]
and Hilbert-Schmidt operators:
\begin{equation}\label{eq: identification}
\bH \cong HS\bigg( \bigotimes_{\mu \in t} H_\mu, \bigotimes_{\nu \notin t} H_\nu \bigg),
\end{equation}
see, e.g.,~\cite{Hackbusch2012}.
In the finite-dimensional case, this simply amounts to reshaping the tensor into a
matrix, with the indices corresponding to $t$ merged into the row indices.

The isomorphism~\eqref{eq: identification} allows us to introduce the \emph{$t$-rank} of $\bu \in \bH$, denoted by $\rank^{(t)}(\bu)$, as the rank of the associated Hilbert-Schmidt operator.
Correspondingly, the sequence of singular values $(\sigma_k^{(t)})$, and the best bilinear approximation errors
\[
\tau^{(t)}_r(\bu) = \min_{\rank^{(t)} (\bv) \le r} \| \bu - \bv \|_{\bH} = \bigg( \sum_{k=r+1}^\infty (\sigma_k^{(t)})^2 \bigg)^{1/2}
\]
can be defined.

Theorem~\ref{th: abstract result d2} implies for fixed $t$ that
\begin{equation}\label{eq:simplified estimate for t-rank}
\tau^{(t)}_r(\bu) \le c \pi_1^{(t)}(\bu) q^{-1}\left( \frac 1 r \right)^{\abs{\frac{\ln q}{\ln \mult^{(t)}}}}
\end{equation}
under slightly modified assumptions. In particular,
the property~\eqref{A2} is replaced by 
\[
\rank^{(t)}(\bu_{n+1}) \le \mult^{(t)}\cdot \rank^{(t)}(\bu_n)  
\]
for some $\mult^{(t)}>0$. The other properties remain the same.
In principle, the constants
$q$ and $c$ involved in~\eqref{A1} 
could also depend on $t$ but, for simplicity, we omit this dependence.
With $\domain$ defined as in~\eqref{eq:domain},
the analogue of the main assumption~\eqref{A4}
is that the quantity
\begin{equation} \label{eq:defpi}
 \pi_1^{(t)}(\bu) = \inf_{\substack{\bv \in \domain \\ \rank^{(t)}(\bv) \le 1}} \|\bv - \bu\|_\bH
\end{equation}
is finite.

Knowing the decay properties of $\tau^{(t)}_r(\bu)$ for certain choices of $t$ is crucial for understanding the approximability of $\bu$ in subspace based low-rank tensor formats.
For instance, the tensor train format~\cite{Oseledets2011} involves the $t$-ranks
of $t=\{1,2,\dots,\mu\}$ for $\mu = 1,2,\dots,d-1$. For prescribed ranks $r_\mu$,
the best approximation error in this format admits the quasi-optimal bound~\cite[Thm. 2.2]{Oseledets2010}
\[
 \sqrt{\big(\tau^{\{1\}}_{r_1}(\bu)\big)^2 + 
 \big(\tau^{\{1,2\}}_{r_2}(\bu)\big)^2 + \cdots + \big(\tau^{\{1,\ldots,d-1\}}_{r_{d-1}}(\bu)\big)^2}.
\]
More specifically, $d$-independent bounds on the von Neumann entropies of the singular values $(\sigma_k^{(t)})$ for these specific choices of $t$ constitute one-dimensional \emph{area laws} in the theory of quantum spin systems~\cite{Arad2012,EisertCramerPlenio2010,Hastings2007}.



\section{Linear equations with low-rank operators and low-rank data} \label{sec:linearequations}


We now apply the general framework from Section~\ref{sec: abstract results}
to a linear system 
\begin{equation}\label{eq: operator equation}
\bA \bu = \bb,
\end{equation}
where $\bA$ is a self-adjoint operator on $\bH$ with 
\begin{equation}\label{eq: coercivity}
\gamma \| \bv \|_\bH^2 \le \langle \bv, \bA \bv \rangle_\bH \le \Gamma \| \bv \|_\bH^2
\end{equation}
for some $0 < \gamma < \Gamma \change{< \infty}$. In particular, this is the case when all Hilbert space are finite-dimensional and $\bA$ is a Hermitian positive definite matrix acting on $\bH$.



The solution $\bu$ of~\eqref{eq: operator equation} is a fixed-point of the \emph{Richardson iteration} 
\begin{equation}\label{eq: Richardson step}
\bu_{n+1} = \Phi(\bu_n) := \bu_n - \alpha (\bA \bu_n - \bb), \quad \alpha = \frac{2}{\gamma + \Gamma}.
\end{equation}
It is well-known that the convergence rate is bounded as follows:
\[
\| \bI - \alpha \bA \|_{\bH \to \bH} \le \frac{\kappa - 1}{\kappa + 1} < 1,
\]
with 
the condition number $\kappa = \Gamma / \gamma$.
Therefore, 
\begin{equation}\label{eq: contraction of Richardson}
\| \bu_{n+1} - \bu \|_\bH \le \left( \frac{\kappa - 1}{\kappa + 1} \right)^{n+1} \| \bu_0 - \bu \|_\bH
\end{equation}
holds for all $n\ge 0$.

%

For a fixed choice of 
$t \subseteq \{1,2,\dots,d\}$, $0< \abs{t} < d$,
we now assume that 
the operator and right-hand side admit a low-rank representation
with respect to the splitting~\eqref{eq: identification}:
\begin{equation}\label{eq: A with low t-rank}
\bA = \sum_{i=1}^{r_\bA^{(t)}} A^{(t)}_i \otimes A^{(t^c)}_i, \qquad 
\bb = \sum_{j = 1}^{r_\bb^{(t)}} b_j^{(t)} \otimes b_j^{(t^c)},
\end{equation}
where $t^c = \{1,2,\dots,d\} \setminus t$. We will assume $r_\bb^{(t)} \le r_\bA^{(t)}$, the general \change{finite rank} case can be obtained by superposition 
\change{as follows. If the $t$-rank of $\bb$ is finite but exceeds $r_\bA^{(t)}$, we first write $\bb = \bb_1 + \cdots + \bb_m$ such that each summand has $t$-rank
at most $r_\bA^{(t)}$. We then apply the result below to each linear system $\bA \bu_1 = \bb_1$, $\ldots$,  $\bA \bu_m = \bb_m$ 
to obtain approximability results for $\bu = \bu_1 + \cdots + \bu_m$.} 

\begin{theorem}\label{th: bound from steepest descent for linear equations}
Given~\eqref{eq: coercivity} and~\eqref{eq: A with low t-rank} with $r_\bb^{(t)} \le r_\bA^{(t)}$, the solution $\bu$ of~\eqref{eq: operator equation} satisfies
\begin{equation} \label{eq:lalabound}
 \tau_r^{(t)}(\bu) \le \frac{\| \bu \|_\bH}{q} \left( \frac 1 r \right)^{\abs{\frac{ \ln q}{\ln R^{(t)}}}} 
\end{equation}
with $R^{(t)} = r_\bA^{(t)}+2$ and $q = \frac{\kappa - 1}{\kappa + 1}$.
If, additionally, $A_i^{(t)}$ or $A_i^{(t^c)}$
in~\eqref{eq: A with low t-rank} is the identity 
for some $i$, then~\eqref{eq:lalabound}
holds with $R^{(t)} = r_\bA^{(t)}+1$.
\end{theorem}

\begin{proof}
By expanding all terms, one concludes from~\eqref{eq: Richardson step} and~\eqref{eq: A with low t-rank} that
\begin{equation}\label{eq: rank increase SD}
\rank^{(t)}(\bu_{n+1}) \le \rank^{(t)}(\bu_n) + r_\bA^{(t)} \rank^{(t)}(\bu_n) + r_\bb^{(t)} \le (r_\bA^{(t)}+2) \rank^{(t)}(\bu_n).
\end{equation}
Taking also~\eqref{eq: contraction of Richardson} into account, we see that for any starting point $\bu_0 \in \bH$ the conditions~\eqref{A1} and~\eqref{A2} hold with $q = \frac{\kappa - 1}{\kappa + 1}$, $c=1$, and $R^{(t)} = r_\bA^{(t)} + 2$. Hence,
the domain $\domain$ considered in~\eqref{eq:domain} can be taken to be $\domain = \bH$
for this choice of parameters, and therefore~\eqref{A4} trivially holds.
Considering $\bu_0 = \mathbf{0}$ yields the estimate $\pi^{(t)}_1(\bu) \le \| \bu \|_\bH$. Consequently, the first part of the theorem is an instance of~\eqref{eq:simplified estimate for t-rank}.

To show the second part, we may assume w.l.o.g. that $A_1^{(t)} = I$ in~\eqref{eq: A with low t-rank}. Then we can rewrite
\[
\bu_n - \alpha \bA \bu_n = \left( I \otimes (I - \alpha A_1^{(t)}) - \alpha \sum_{i=2}^{r_\bA^{(t)}} A^{(t)}_i \otimes A^{(t^c)}_i \right) \bu_n,
\]
so that the rank actually increases at most by a factor of $R^{(t)} = r_\bA^{(t)} + 1$. 
\end{proof}

\begin{example}\label{example: nearest neighbor interaction}
The following structure occurs frequently in applications of high-dimensional operator equations:
\begin{equation} \label{eq:structureA}
  \mathbf A = \mathbf L + \mathbf V,
\end{equation}
where
\begin{gather*}
  \mathbf L = A_1 \otimes I \otimes \cdots \otimes I + I \otimes A_2 \otimes \cdots \otimes I + \cdots + I \otimes  \cdots \otimes I \otimes A_d,\\
 \mathbf V = B_1 \otimes C_2 \otimes I \otimes \cdots \otimes I + 
 I \otimes B_2 \otimes C_3 \otimes \cdots \otimes I +
 I \otimes \cdots \otimes I \otimes B_{d-1} \otimes C_{d}.
\end{gather*}
Here, the $\mu$th term of $\mathbf L$ represents the action on the $\mu$th variable. For example,
a structured discretization of the $d$-dimensional Laplace operator takes this form. The terms in $\mathbf V$ describe interactions between two neighboring variables.

We assume that all involved coefficients $A_\mu$, $B_\mu$, and $C_\mu$ are bounded self-adjoint
operators satisfying the inequalities
\[\gamma_A \le A_\mu \le \Gamma_A, \quad 
0 \le B_\mu \le \Gamma_B, \quad 
0 \le C_\mu \le \Gamma_C\]
in the spectral sense, for some constants $\gamma_A, \Gamma_A, \Gamma_B, \Gamma_C>0$
independent of $\mu$.
Then $\mathbf A$ is a bounded self-adjoint operator satisfying the inequality~\eqref{eq: coercivity}
with $\gamma = d \gamma_A$ and $\Gamma = d\Gamma_A + (d-1) \Gamma_B \Gamma_C$.
Consequently, the condition number $\kappa$ determining the contraction rate~\eqref{eq: contraction of Richardson} is bounded independently of $d$.

On the other hand, it can be shown by an explicit construction~\cite{Khoromskij2010d,KreSU13} that any operator having the algebraic structure~\eqref{eq:structureA} admits a low-rank representation of the form~\eqref{eq: A with low t-rank}
with $r_\bA^{(t)} = 3$ for any $t = \{1,2,\ldots,\mu\}$.
In turn, the solution to an operator equation with the structure in~\eqref{eq:structureA}
and low-rank right-hand side $\bb$ satisfies the decay estimate~\eqref{eq:lalabound}
for any such $t$, independently of $d$. As discussed at the end of Section~\ref{sec: abstract results},
this implies $d$-independent approximability in the tensor train format.
By~\cite[Ex. 5.2]{Kressner2011a}, the same conclusion holds for the hierarchical Tucker format.
\end{example}

\change{
It is instructive to discuss the special case $\mathbf V = 0$ in Example~\ref{example: nearest neighbor interaction}, which corresponds to the absence of the neighbor interaction terms $B_\mu$ and $C_\mu$. Resolving the recursion, the iterates produced by the method of steepest descent~\eqref{eq: Richardson step} take the form
\begin{equation} \label{eq:explicitrepresentation}
 \bu_{n} = (\bI - \alpha \bL)^n \bu_0 + \alpha \sum_{i = 0}^{n-1} (\bI - \alpha \bL)^i \bb.
\end{equation}
For a fixed choice of $t \subseteq \{1,2,\dots,d\}$, $0< \abs{t} < d$,
the structure of $\bL$ implies that we can partition, similarly as in~\eqref{eq: A with low t-rank},
\[
 \bI - \alpha \bL = I \otimes L^{(t)} + L^{(t^c)} \otimes I.
\]
Noting that $I \otimes L^{(t)}$ and $L^{(t^c)} \otimes I$ commute,
this allows to rewrite~\eqref{eq:explicitrepresentation} as
\[
 \bu_{n} = p( L^{(t)}, L^{(t^c)}) \bu_0 + \alpha q( L^{(t)}, L^{(t^c)}) \bb,
\]
with
\[
 p(L^{(t)}, L^{(t^c)} ) = \sum_{k=0}^n \binom{n}{k} (L^{(t)})^{k} \otimes  (L^{(t^c)})^{n-k}
\]
and
\[
 q(L^{(t)}, L^{(t^c)} ) = \sum_{\ell=0}^{n-1} \sum_{k=0}^{\ell} \binom{\ell}{k} (L^{(t)})^{k} \otimes  (L^{(t^c)})^{\ell-k} = \sum_{k=0}^{n-1} (L^{(t)})^{k} \otimes  \left( \sum_{\ell=k}^{n-1} \binom{\ell}{k}  (L^{(t^c)})^{\ell-k}\right)
\]
Combined with~\eqref{eq:explicitrepresentation}, this implies
\[
 \rank^{(t)}(\bu_{n}) \le (n+1) \rank^{(t)}(\bu_0) + n \rank^{(t)}(\bb).
\]
This allows us to replace 
the error estimate~\eqref{eq: estimate for powers of mult} in the proof of Theorem~\ref{th: abstract result d2} 
}
by $(\tau_n^{(t)})^2 \lesssim (\frac{\kappa - 1}{\kappa +1 })^{2n}$. In turn, we obtain exponential singular value decays with respect to all such $t$.
Similar and even stronger results can be obtained by approximating the inverse $\bL^{-1}$ of the Laplace-like operator $\bL$ by exponential sums~\cite{Grasedyck2004,Hackbusch2012}.

\section{Eigenvalue problems with low-rank operators}\label{sec: eigenvectors}

As another application of our general framework, we now consider the approximability
of an eigenvector $\bu$ belonging to the smallest eigenvalue $\lambda_1$ of a 
bounded self-adjoint operator $\bA\vcentcolon\bH\to\bH$.
In particular, we have 
\begin{equation}\label{eq:spectrum bounds}
\lambda_1 \| \bv \|^2 \le \langle \bv, \bA \bv \rangle_\bH^{} \le \Gamma \| \bv \|^2_{\bH},
\end{equation}
for some $\Gamma$.

In the following, we assume $\lambda_1$ to be simple. This implies that 
the rest of the spectrum is contained in an interval $[\lambda_2, \Gamma]$ with $\lambda_2 > \lambda_1$. The \emph{absolute gap} and the \emph{relative gap} are denoted by
\begin{equation} \label{eq: gaps}
 \delta = \lambda_2 - \lambda_1, \qquad \Delta = \frac{\delta}{\Gamma - \lambda_1},
\end{equation}
respectively.
These gaps play a critical role in our estimates.


We now fix $\bu$, and denote by $\langle \bu \rangle$ the linear span of $\bu$. To approximate $\bu$, we apply the Richardson iteration to the singular linear system
\begin{equation}\label{equivalent linear system}
\bA_{\lambda_1} \bu := (\bA - \lambda_1 \mathbf{I}) \bu = \mathbf{0},
\end{equation}
but on the nontrivial invariant subspace $\langle \bu \rangle^\bot$. This results in the iteration
\begin{equation}\label{eq: Richardson iteration for ground state}
\bu_{n+1} = \Phi(\bu_n) := \bu_n - \beta \bA_{\lambda_1} \bu_n =  (1 + \beta \lambda_1) \bu_n - \beta \bA \bu_n, \quad \beta = \frac{2}{\delta + \Gamma - \lambda_1}, \quad \bu_0 \in \bu + \langle \bu \rangle^\bot.
\end{equation}
We emphasize that this method assumes the knowledge of the exact $\lambda_1$ a priori. It is therefore primarily of theoretical interest, to derive the desired error estimates for the low-rank approximability of the eigenvector $\bu$. In turn, these estimates could be used to design
a practical method of optimal complexity, in the spirit of~\cite{BachmayrDahmen2015}.

In order to apply Theorem~\ref{th: abstract result d2},
we now verify that the properties~\eqref{A1} and~\eqref{A2} are satisfied.
We begin with discussing the convergence of~\eqref{eq: Richardson iteration for ground state}. By the simplicity of $\lambda_1$, the self-adjoint operator $\bA_{\lambda_1} = \bA - \lambda_1 \bI$ has the one-dimensional kernel $\langle \bu \rangle$. It is bounded from below and above by $\delta$ and $\Gamma -\lambda_1$, respectively, on the invariant subspace $\langle \bu \rangle^\bot$, so its condition number on this subspace is \change{bounded by} $1/\Delta$. This implies that the spectral radius of $\bI - \beta \bA_{\lambda_1}$ on the invariant subspace $\langle \bu \rangle^\bot$ is bounded by $\frac{1-\Delta}{1+\Delta}$. Since
\(
\bu_{n+1} - \bu = (\bI - \beta \bA_{\lambda_1})(\bu_n - \bu),
\)
an induction shows that if $\bu_0 - \bu \in \langle \bu \rangle^\bot$, then $\bu_n - \bu \in \langle \bu \rangle^\bot$ for all $n$, and
\begin{equation}\label{eq: contraction of ground state Richardson}
\| \bu_{n+1} - \bu \|_\bH \le \left( \frac{1-\Delta}{1+\Delta} \right)^{n+1} \| \bu_0 - \bu \|_\bH \quad \text{if $\bu_0 \in \bu + \langle \bu \rangle^\bot$.}
\end{equation}
In other words,~\eqref{A1} holds with $q = \frac{1 - \Delta}{1 + \Delta}$.
%

As for ~\eqref{A2}, similar to~\eqref{eq: rank increase SD}, the $t$-ranks of the iteration~\eqref{eq: Richardson iteration for ground state} satisfy
\begin{equation}\label{eq: rank estimate for groundstate SD}
\rank^{(t)}(\bu_{n+1}) \le (r_\bA^{(t)}+1) \rank^{(t)}(\bu_n),
\end{equation}
provided that $\bA$ admits a representation of the form~\eqref{eq: A with low t-rank}.
Once again, if one of the operators $A_i^{(t)}$ or $A_i^{(t^c)}$ in~\eqref{eq: A with low t-rank} is the identity, then $r_\bA^{(t)} + 1$ can be replaced by $r_\bA^{(t)}$ in~\eqref{eq: rank estimate for groundstate SD}. In both cases, property~\eqref{A2} is satisfied.


\begin{paragraph}{\bf Assumption (A0).} By~\eqref{eq: contraction of ground state Richardson}, 
the set $\domain$ defined in~\eqref{eq:domain} takes the form
\[
 \domain = \{ \bv \in \bH \vcentcolon \langle \bv - \bu , \bu \rangle_\bH = 0 \}.
\]
To verify the main assumption~\eqref{A4}, we have to show
that $\domain$ contains a starting point having $t$-rank at most one.
In fact, let $\hbu_0$ be any element with $\rank^{(t)}(\hbu_0) = 1$ that is not orthogonal to $\bu$. Then
\begin{equation}\label{qe: rescaled rank-one starting point}
\bu_0 = \frac{\| \bu \|_\bH^2}{\langle \bu, \hbu_0 \rangle_\bH} \hbu_0 \in \domain
\end{equation}
with $\rank^{(t)}(\bu_0) = 1$. In turn, the quantity $\pi_1^{(t)}(\bu)$ defined in~\eqref{eq:defpi} is finite.
\end{paragraph}

Our findings above allow us to apply Theorem~\ref{th: abstract result d2} for estimating
the $t$-rank approximation error of the eigenvector $\bu$.
\begin{theorem}\label{th: EV problems}
Given~\eqref{eq: A with low t-rank} and~\eqref{eq:spectrum bounds}, the solution $\bu$ of~\eqref{equivalent linear system} satisfies
\begin{equation}\label{eq: decay rate for EV}
\tau_r^{(t)}(\bu) \le \frac{\pi^{(t)}_1(\bu)}{q}  \left( \frac 1 r \right)^{\abs{\frac{ \ln q}{\ln R^{(t)}}}},
\end{equation}
with $q = \frac{1 - \Delta}{1 + \Delta}$, $R^{(t)} = r_\bA^{(t)}+1$, and the gaps $\delta,\Delta$ defined in~\eqref{eq: gaps}.
If, additionally, $A_i^{(t)}$ or $A_i^{(t^c)}$
in~\eqref{eq: A with low t-rank} is the identity 
for some $i$, then~\eqref{eq: decay rate for EV}
holds with $R^{(t)} = r_\bA^{(t)}$.
\end{theorem}

A notable difference of Theorem~\ref{th: EV problems} 
to Theorem~\ref{th: bound from steepest descent for linear equations} is that it features the quantity $\pi^{(t)}_1(\bu)$ in the estimate. This quantity 
measures the distance between $\bu$ and the set of $t$-rank one tensors within the affine space $\bu + \langle \bu \rangle^\bot$. In this way, the problem of rank-$r$ approximability has been reduced to the problem of rank-one approximability.

\subsection{The problem of $t$-rank one approximability}

In this section, we derive upper bounds for the quantity $\pi^{(t)}_1(\bu)$ defined in~\eqref{eq:defpi}.
Trivially, every starting point $\bu_0 \in \domain$ of $t$-rank one yields the upper bound
$\|\bu_0 - \bu\|_{\bH}$. While this is of interest when considering a specific iteration,
more insight would be gained from bounds that depend on $\delta$, $\Delta$, and $r_\bA^{(t)}$ only.
Deriving such bounds is surprisingly difficult and at the heart of related works on the entanglement entropy, see, e.g.,~\cite{Arad2012}.


In an infinite-dimensional tensor product space $\bH$, the ratio $\pi^{(t)}_1(\bu) / \| \bu \|_\bH$ may become arbitrarily large
for arbitrary $\bu \in \bH$.
Upper bounds are obtained from $t$-rank one approximations to $\bu$ in the $\bH$-norm. Specifically, considering~\eqref{qe: rescaled rank-one starting point} with $\|\hbu_0\|_\bH = 1$, we get the estimate
\begin{equation*}
 \pi^{(t)}_1(\bu) \le \| \bu_0 -\bu \|_\bH \le \|\bu_0\|_\bH =  \frac{\| \bu \|_\bH}{\abs{\left \langle \frac{\bu}{\| \bu\|_\bH} , \hbu_0 \right\rangle_{\bH}}} ,
\end{equation*}
where we used that $\bu_0 -\bu$ is orthogonal to $\bu$.
Thus, the problem is further reduced to providing a lower bound on the overlap of the normalized eigenvector with normalized tensors of $t$-rank one:
\begin{equation}\label{eq: reduction to overlap}
\pi_1^{(t)}(\bu) \le \frac{\| \bu \|_\bH}{\theta_1^{(t)}(\bu)},
\end{equation}
where
\begin{equation}\label{definition of theta}
\theta_1^{(t)}(\bu) := \sup_{\substack{\rank^{(t)} (\hbu_0) = 1 \\ \| \hbu_0 \|_\bH = 1}} \left \langle \frac{\bu}{\| \bu\|_\bH} , \hbu_0 \right\rangle_{\bH}.
\end{equation}

In the case that every $H_\mu$ has finite dimension $N_\mu$, $\mu=1,\dots,d$, a generic bound is obtained as follows. The singular value decomposition~\eqref{eq:SVD} of the solution with respect to the identification~\eqref{eq: identification} is a finite sum with
\[
D^{(t)} = \min\bigg( \prod_{\mu \in t} N_\mu, \prod_{\nu \notin t} N_\nu \bigg)
\]
mutually orthogonal $t$-rank one tensors of decreasing norms $\sigma_k^{(t)}$. This implies that the overlap~\eqref{definition of theta}
is at least $\sigma_1^{(t)} / \| \bu\|_\bH \ge 1/\sqrt{D^{(t)}}$. By~\eqref{eq: reduction to overlap}, we obtain
\begin{equation}\label{eq: naive rank-one overlap estimate}
\pi^{(t)}_1(\bu) \le \sqrt{ D^{(t)} } \| \bu \|_\bH.
\end{equation}
This bound is independent of $d$ only when the cardinality of $t$ does not grow, which is the case for 
the Tucker format~\cite{Hackbusch2012}. The tensor train and hierarchical Tucker formats, however, require to take
large splittings like $t = \{1,\dots,d/2\}$ into consideration. Consequently, the bound~\eqref{eq: naive rank-one overlap estimate} grows exponentially with $d$. In~\cite{Arad2012}, one of the very few results on this question, it has been shown how this growth can be avoided in the case of frustration-free systems. This constitutes a rather limiting assumption.
The following result adapts a technique from~\cite[Lemma III.2]{Arad2012},
which does not require this assumption but instead assumes a rather strong contraction relative to the rank growth.
\begin{theorem}\label{th: estimate the overlap}
With the notation introduced in Theorem~\ref{th: EV problems}, assume that
$
 q^2 \mult^{(t)} < 1.
$
Then it holds
\[
\big(\theta_1^{(t)} (\bu)\big)^2 \ge \frac{1}{2} \bigg( \frac{1}{\mult^{(t)}} \bigg)^{\Big\lceil \frac{- \ln 2}{\ln\left(q^2 \mult^{(t)}\right)} \Big\rceil}
\]
for $\theta_1^{(t)} (\bu)$ defined in~\eqref{definition of theta}. Consequently, by~\eqref{eq: decay rate for EV} and~\eqref{eq: reduction to overlap},
\[
\tau_r^{(t)}(\bu) \le \sqrt{2} (\mult^{(t)})^{\frac{1}{2}\Big\lceil \frac{- \ln 2}{\ln\left(q^2 \mult^{(t)}\right)} \Big\rceil} \| \bu \|_\bH \left( \frac 1 r \right)^{\abs{\frac{ \ln q}{\ln \mult^{(t)}}}}.
\]
\end{theorem}
\begin{proof}
Without loss of generality, we may assume $\| \bu \|_\bH = 1$. Let $P$ denote the orthogonal projection onto $\langle \bu \rangle$. To simplify the notation, we write $\theta$ instead of $\theta_1^{(t)}(\bu)$.

Let $\epsilon > 0$ and $\hbu_0$ be an normalized rank-one tensor with $\| P \hbu_0 \|_\bH = \langle \bu , \hbu_0 \rangle_{\bH} \ge \theta - \epsilon$. We let $\hbu_n$ denote the iterate obtained after $n$ steps of the Richardson method~\eqref{eq: Richardson iteration for ground state} with starting vector $\hbu_0$. Since $\hat \bu_0 \in P \hat \bu_0+ \langle \bu \rangle^\bot$, this rescaled
Richardson method converges to $P \hat \bu_0 \not=0$ and, by induction,
\begin{equation}\label{eq: constant overlap}
P \hbu_n = P \hbu_0.
\end{equation}
By~\eqref{eq: contraction of ground state Richardson} and using $\| \hbu_0 \|_\bH = 1$,
\[
\|(I - P)\hbu_n \|_\bH^2 \le q^{2n} \|(I - P) \hbu_0 \|_\bH^2 = q^{2n}(1 - \| P \hbu_0\|_\bH^2).
\]
Hence,
\begin{align}\label{eq: norm of hbun}
\| \hbu_n \|_\bH^2
&= \| P \hbu_n \|_\bH^2 + \| (I - P) \hbu_n \|_\bH^2 = \| P \hbu_0 \|_\bH^2 + \| (I - P) \hbu_n \|^2_\bH\notag \\
&\le \| P\hbu_0 \|_\bH^2 + q^{2n}(1 - \| P \hbu_0\|_\bH^2)\notag \\
&\le \theta^2 +  q^{2n} (1 - (\theta - \epsilon)^2),
\end{align}
where we used that $\| P \hbu_0 \|_{\bH} \le \theta$ by definition~\eqref{definition of theta} of $\theta$.

Using the singular value decomposition, we can write
\[
\hbu_n = \sum_{k=1}^{\rank^{(t)}(\hbu_n)} \sigma_k \bv_k,
\]
with mutually orthonormal $t$-rank one tensors $\bv_k$. By the Cauchy-Schwarz inequality,
\[
(\theta - \epsilon)^2  \le \abs{\langle \bu, \hbu_0 \rangle_\bH}^2 = \abs{\langle \bu, \hbu_n \rangle_\bH}^2  \le \bigg(\sum_{k=1}^{\rank^{(t)}(\hbu_n)} \abs{ \langle \bu, \bv_k \rangle_\bH }^2 \bigg) \| \hbu_n \|_\bH^2,
\]
where the equality follows from~\eqref{eq: constant overlap}. As $\rank^{(t)}(\hbu_n) \le (\mult^{(t)})^n$, we conclude using~\eqref{eq: norm of hbun} that 
\begin{equation*}
\theta^2 \ge | \langle \bu, \bv_k \rangle_\bH |^2 \ge \frac{(\theta - \epsilon)^2}{(\mult^{(t)})^n\| \hbu_n \|_{\bH}^2} \ge \frac{(\theta - \epsilon)^2}{(\mult^{(t)})^n(\theta^2 + q^{2n} (1 - (\theta - \epsilon)^2))}
\end{equation*}
holds for at least one $k$.
Note that the first inequality again is due to the definition of $\theta$. As $\epsilon$ can be chosen arbitrary, we obtain
\[
(\mult^{(t)})^n(\theta^2 + q^{2n} (1 - \theta^2)) \ge 1,
\]
or, equivalently,
\begin{equation}\label{eq:equivalently}
\theta^2 (1 - q^{2n}) \ge (\mult^{(t)})^{-n} - q^{2n} = (\mult^{(t)})^{-n} (1 -  (q^2\mult^{(t)})^n).
\end{equation}

For $n \ge \frac{-\ln 2}{\ln(q^2\mult^{(t)})}$, which is positive by assumption, we have $(q^2\mult^{(t)})^n \le 1/2$. Then~\eqref{eq:equivalently} implies
\begin{equation}\label{eq:final estimate}
\theta^2 \ge \frac{1}{(\mult^{(t)})^{n}} \frac{1 - (q^2\mult^{(t)})^n}{1 - q^{2n}} \ge \frac{1}{2(\mult^{(t)})^{n}}.
\end{equation}
The assertion follows by choosing $n = \left\lceil \frac{- \ln 2}{\ln(q^2\mult^{(t)})} \right\rceil$.
\end{proof}
Note that better bounds on $\theta$ may be obtained from~\eqref{eq:final estimate} by estimating the maximum value of the middle term as a function of $n$ more carefully, but this quickly becomes clumsy.

The proof of Theorem~\ref{th: estimate the overlap} is based on the intuition that the ratio between the energy contraction rate $q^{2n}$ and the reciprocal rank increase $1/(\mult^{(t)})^n$ after $n$ steps of the Richardson iteration can be made arbitrarily small when $q^2 \mult^{(t)} < 1$. Interestingly, this assumption alone does not result in better singular value decays in any of the above theorems, as only the ratio of the logarithms enters.
The consideration of several steps of the fixed-point iteration only pays off when improved estimates of $\mult^{(t)}$ are available,
as discussed for linear systems \change{at the end of Section~\ref{sec:linearequations}}. An example of relevance to eigenvalue problems is given, for instance, by an operator of the form
\[
\bA = A_1 \otimes I + I \otimes A_2 + B \otimes C,
\]
see also Example~\ref{example: nearest neighbor interaction}.
A direct calculation reveals that for such an operator
two steps of steepest descent~\eqref{eq: Richardson iteration for ground state} do not increase the rank by a factor of $3^2 = 9$, but only by at most $6$.


\section{Conclusions}

We have established bounds on the singular value decays for solutions to tensor structured linear systems and eigenvalue problems. As these decays
govern the low-rank approximability in various low-rank tensor formats, such as the tensor train and the hierarchical Tucker formats,
our results allow to make a priori statements about the suitability of these formats to address a given application, possibly even for large orders $d$.

With the assumptions made in this paper, our construction yields algebraic decays. To obtain exponential decays, as they are sometimes
observed in practice, further assumptions may be needed. In \change{Section~\ref{sec:linearequations}}, a rather restrictive commutativity assumption
is shown to yield exponential decays. It would certainly be of interest to identify less restrictive assumptions.

\section*{Acknowledgment}

We thank Markus Bachmayr and Bart Vandereycken for inspiring discussions on an earlier draft of this paper, which resulted in some valuable improvements.

\bibliographystyle{plain}      
\bibliography{LowRankApprox}   


\end{document}